\documentclass[10pt,reqno]{amsart}
\author[H.~Kawanoue]{Hiraku Kawanoue} 
\address{
College of Science and Engineering, Chubu University
\newline 
Matsumoto-cho, Kasugai-shi, 
Aichi 487-8501, JAPAN
} 
\email{kawanoue@kurims.kyoto-u.ac.jp} 
\title{On 
the extended Catalan arrangement of type $B_\ell$}

\newcommand{\bZ}{\mathbb{Z}}
\newcommand{\bC}{\mathbb{C}}

\newcommand{\cA}{\mathcal{A}}
\newcommand{\cat}{\operatorname{Cat}}

\theoremstyle{plain}
\newtheorem{thm}{Theorem}
\newtheorem{lem}[thm]{Lemma}
\newtheorem{cor}[thm]{Corollary}
\newtheorem{prop}[thm]{Proposition}
\theoremstyle{Remark}

\usepackage{version,graphics,url,txfonts,times,bm,color}
\begin{document}
\begin{abstract}
We give an explicit formula for the generators of the logarithmic 
vector field of the coning of the extended Catalan arrangement 
of type $B_\ell$.
\end{abstract}
\maketitle
%\tableofcontents
\begin{section}{Introduction}
The extended Catalan arrangement is a class of finite truncation of 
an affine Weyl arrangement for a root system. 
In 2004, Yoshinaga \cite{Y} proved that the conings of the 
extended Catalan arrangements are free 
for all root systems. 
However, the explicit description of 
their logarithmic vector field are not well known 
for a long time.  

In 2021, Suyama and Yoshinaga \cite{SY} gave 
an explicit basis of the extended Catalan arrangement of type $A_\ell$.
In 2023, Feigin, Wang and Yoshinaga constructed 
a set of derivations in view of their integral expression 
of the basis of Coxeter arrangement of type $B_2$, and 
asked if it really gives 
a basis of the extended Catalan arrangement of type $B_2$ 
or not. 

In this article, we generalize their construction to 
the extended Catalan arrangement of type $B_\ell$ 
and confirm that they give a basis.
In particular, it settles the conjecture of 
Feigin, Wang and Yoshinaga in \cite{FWY} affirmatively .

\bigskip

We work over the complex number field $\mathbb C$.
We denote 
$$(a)^\downarrow_c=
\begin{cases}
\prod_{i=a-c+1}^ai & (c\geq0)
\\
\left\{(a-c)^\downarrow_{-c}\right\}^{-1}
& (c\leq0)
\end{cases},
\quad 
(b)^\uparrow_c=(b+c-1)^\downarrow_c,
\quad 
P(a,b)=(a+b)^\downarrow_{2b+1}.
$$
Let $\ell\in\bZ_{>0}$ and $m\in\bZ_{\geq0}$.  
The extended Catalan arrangement 
$\cat(B_\ell,m)$ of type $B_\ell$ is 
a hyperplane arrangement in $\bC^\ell$
defined by the following equation
$$
\Psi_{\ell,m}=\prod_{1\leq a\leq \ell}
P(x_a,m)
\prod_{1\leq b<c\leq \ell}
P(x_b+x_c,m)
P(x_b-x_c,m)
=0.
$$
We also consider 
a central arrangement $c\!\,\cat(B_\ell,m)$ in $\bC^{\ell+1}$ 
defined by $z\widetilde{\Psi}_{\ell,m}=0$, where 
$\widetilde{\ }$ means homogenization of a polynomial by 
supplying a new variable $z$.
It is called the coning of $\cat(B_\ell,m)$.
As stated above, it is known that 
$c\!\,\cat(B_\ell,m)$ is free.

For 
$\ell,m,u\in\bZ_{>0}$, we define 
\begin{align*}
\lefteqn{
h_{\ell,m,u}(x;\bm{y})=
h_{\ell,m,u}(x;y_1,y_2,\dotsc,y_{\ell})=
}
\\
&\sum_{0\leq s_1,s_2,\dotsc,s_\ell\leq m}
\dfrac{
(-1)^{S_\ell}P(x,m+u+S_\ell)
}
{
\binom{ m+u+S_\ell + \frac12}{m + 1}
}
\prod_{i=1}^\ell
\binom{m}{s_i}
\dfrac{
P(y_i,u + m+S_{i-1})
}{
P(y_i,u + S_{i})
}
,\end{align*}
where $S_k=\sum_{j=1}^ks_j$.  
For $\ell=0$, we interpret $h_{0,m,u}(x)$ as 
$$
h_{0,m,u}(x) =
\dfrac{P(x,m+u)
}{\binom{ m+u + \frac12}{m + 1}}.
$$
We define 
$$
F_{\ell,m,u,i}(x_1,\dotsc,x_\ell)=
h_{\ell-1,m,u}(x_i;x_1,\dotsc,\hat{x_{i}},\dotsc,x_{\ell})
\quad(1\leq i\leq \ell),
$$
where $\hat{x_i}$ means the omission of $x_i$, and 
$$\delta_{\ell,m,u}=
\sum_{i=1}^\ell \widetilde{F}_{\ell,m,u,i}(x_1,\dotsc,x_\ell)\partial_{ x_i}.
$$
We denote the Euler derivation as 
$\delta_E=z\partial_z+\sum_{i=1}^\ell x_i\partial_{x_i}$.
The following assertion is the main theorem of this article.
\begin{thm}\label{Main}
The set 
$\{\delta_E\}\cup
\{\delta_{\ell,m,u}\mid u=0,1,\dotsc,\ell-1\}$
forms a basis of 
$D(c\!\,\cat(B_\ell,m))$. 
\end{thm}

The conjectural basis of $D(c\!\,\cat(B_2,m))$ 
constructed by Feigin, Wang and Yoshinaga \cite{FWY} 
is based on considering 
the discretization $\widetilde{f_u^m}(x,y)$ of the integral
$\int_0^xt^{2u}(t^2-x^2)^m(t^2-y^2)^m\,dt$.
Our construction is generalization of theirs.  
In particular, 
our $h_{1,m,u}$ corresponds to 
their $\widetilde{f_u^m}$.  Note that 
\begin{align*}
\lefteqn{
\int_0^xt^{2u}(t^2-x^2)^m(t^2-y^2)^m\,dt
=x^{4m+2u+1}\int_0^1r^{u}(r-1)^m
(r-x^{-2}y^2)^m\,\dfrac{dr}{2\sqrt{r}}
}
\\
&=\dfrac12x^{4m+2u+1}\sum_{s=0}^m
\binom{m}{s}
(-1)^{s}(x^{-1}y)^{2(m-s)}
\int_0^1
r^{s+u-\frac12}
(1-r)^m
\,dr
\\
&=\dfrac12\sum_{s=0}^m
\binom{m}{s}
(-1)^{s}
x^{2m+2u+2s+1}
y^{2m-2s}
B\Bigl(s+u+\dfrac12,m+1\Bigr)
\\
&
=\dfrac1{2(m+1)}
\sum_{s=0}^m
\dfrac{
(-1)^{s}\binom{m}{s}
}{\binom{m+u+s+\frac12}{m+1}}
x^{2m+2u+2s+1}y^{2m-2s}.
\hfill\tag{1}
\end{align*}
This equation implies $h_{1,m,u}=2(m+1)\widetilde{f_u^m}$
by comparing the construction in \cite{FWY} to ours.  Thus 
Theorem \ref{Main} 
settles the conjecture in \S5 of  \cite{FWY} 
affirmatively.
\begin{cor}
The conjecture of Feigin, Wang and Yoshinaga in \cite{FWY} is true.  
Namely, their candidate actually gives a basis 
of $D(c\!\,\cat(B_2,m))$.
\end{cor}
\begin{subsection}*{Acknowledgments}
We are grateful to T. Abe and M. Yoshinaga for their interest to 
our work.  We thank T. Nozawa for discussing always and 
reading the manuscript carefully.

\end{subsection}
 
\end{section}
\begin{section}{Preliminaries}
We investigate the properties of 
$h_{\ell,m,u}(x;\bm{y})$ which will be used in the 
proof of the main theorem.

\begin{subsection}{The case $\ell=1$}
%First we see the properties for the case of $\ell=1$. 
We give two formulae for $h_{1,m,u}$.  Recall that 
$$
h_{1,m, u}( x; y)=
\sum_{s=0}^m
\dfrac{(-1)^{s}\binom{m}{s}}{\binom{ m+u+s + \frac12}{m + 1}}
P(x,m+u+s)\dfrac{P(y,u + m)}{P(y,u + s)}.
$$

\begin{prop}\label{prop1}
For $m\in\bZ_{\geq0}$ and $u\in\bZ_{\geq1}$, the following 
recurrence formula for $m$ holds.
$$
\dfrac{2u - 1}{2m + 4}
h_{1,m + 1, u - 1}( x; y)
=\left\{
x^2 + y^2 - (m + u + 1)^2 - u^2\right\}
h_{1,m, u}( x; y) - 2h_{1,m, u + 1}( x; y)
.$$
\end{prop}
\begin{proof}
We denote that 
$M_{r}=A_1^{x-1}A_2^{-x-1}B_1^{y+r}B_2^{-y+r}$ 
with formal variables $A_\ast$, $B_\ast$, and that 
$$G_{m,u}(X,Y)=
\sum_{s=0}^m
\dfrac{(-1)^{s}\binom{m}{s}}{\binom{ m+u+s + \frac12}{m + 1}}
X^{m+u+s}Y^{m-s},
\quad
L_{m,u}=(-1)^u
G_{m,u}(\partial_{A_1}\partial_{A_2},\partial_{B_1}\partial_{B_2}).
$$
We denote by $\overline{\phantom{A}}$ 
the evaluation at $A_\ast=B_\ast=1$.
Then it is easy to see that 
$$h_{1,m,u}=x\,\overline{L_{m,u}M_{m+u}}.$$

On the other hand, by $(1)$, we have 
$$\int_0^xt^{2u}(t^2-x^2)^m(t^2-y^2)^m\,dt=
\dfrac{x}{2(m+1)}G_{m,u}(x^2,y^2).$$
Using integration by parts, for $m\in\bZ_{\geq0}$, we have 
$$
\int_0^xt^{2u-2}(t^2-x^2)^{m+1}(t^2-y^2)^{m+1}\,dt
%=\dfrac{-1}{2u-1}\int_0^xt^{2u-1}\left\{
%(t^2-x^2)^{m+1}(t^2-y^2)^{m+1}\right\}'\,dt
=\dfrac{2(m+1)}{2u-1}\int_0^xt^{2u}
(t^2-x^2)^{m}(t^2-y^2)^{m}
(-2t^2+x^2+y^2)\,dt.
$$
Thus we obtain 
$$
\dfrac{2u-1}{2m+4}G_{m+1,u-1}(X,Y)=
-2G_{m,u+1}(X,Y)+(X+Y)G_{m,u}
$$
and hence 
$$
\dfrac{2u-1}{2m+4}L_{m+1,u-1}=
-2L_{m,u+1}-
(\partial_{A_1}\partial_{A_2}+\partial_{B_1}\partial_{B_2})
L_{m,u}
.\eqno{(2)}$$

We calculate that 
\begin{align*} 
\lefteqn{
\overline{L_{m, u + 1}(M_{m+u+1}-M_{m+u})}
}
\\
&
=\sum_{s=0}^m
\dfrac{(-1)^{s}\binom{m}{s}}{\binom{ m+u+s + \frac32}{m + 1}}
P(x,m+u+s+1)\left\{
\dfrac{P(y,u + m+1)}{P(y,u + s+1)}
-
\dfrac{P(y,u + m)}{P(y,u + s)}
\right\}
\\
&
=\sum_{s=1}^{m+1}
\dfrac{(-1)^{s-1}\binom{m}{s-1}}{\binom{ m+u+s + \frac12}{m + 1}}
P(x,m+u+s)\left\{
\dfrac{P(y,u + m+1)}{P(y,u + s)}
-
\dfrac{P(y,u + m)}{P(y,u + s-1)}
\right\}
\\ &
=\sum_{s=0}^{m}
\dfrac{(-1)^{s}\binom{m}{s}}{\binom{ m+u+s + \frac12}{m + 1}}
P(x,m+u+s)\dfrac{P(y,u + m)}{P(y,u + s)}
(2u+s+m+1)s,
\tag{3}
\end{align*}
and that 
\begin{align*} 
\lefteqn{
-\overline{
(\partial_{A_1}\partial_{A_2}+\partial_{B_1}\partial_{B_2})
L_{m,u}M_{m+u}
}
}
\\ &
=\sum_{s=0}^m
\dfrac{(-1)^{s}\binom{m}{s}}{\binom{ m+u+s + \frac12}{m + 1}}P(y,u+m)
\left\{\dfrac{P(x,m+u+s+1)}{P(y,u + s)}
+\dfrac{P(x,m+u+s)}{P(y,u + s-1)}\right\}
\\ &
=\sum_{s=0}^m
\dfrac{(-1)^{s}\binom{m}{s}}{\binom{ m+u+s + \frac12}{m + 1}}
P(x,m+u+s)\dfrac{P(y,u+m)}{P(y,u + s)}\left\{
x^2+y^2-(m+u+s+1)^2-(u+s)^2
\right\}.
\tag{4}
\end{align*}
Thus we conclude that 
\begin{align*}
&
\dfrac1x\left(
\dfrac{2u - 1}{2m + 4}
h_{1,m + 1, u - 1}
+ 2h_{1,m, u + 1}
\right)
=
\dfrac{2u - 1}{2m + 4}
\overline{L_{m+1,u-1}M_{m+u}}
+ 2
\overline{L_{m,u+1}M_{m+u+1}}
\\ &
\overset{(2)}{=}
 \overline{\{
-2L_{m,u+1}-
(\partial_{A_1}\partial_{A_2}+\partial_{B_1}\partial_{B_2})
L_{m,u}\}M_{m+u}}
+ 2
\overline{L_{m,u+1}M_{m+u+1}}
\\ &
\overset{(3)(4)}{=}
\sum_{s=0}^{m}
\dfrac{(-1)^{s}\binom{m}{s}}{\binom{ m+u+s + \frac12}{m + 1}}
P(x,m+u+s)\dfrac{P(y,u + m)}{P(y,u + s)}
\cdot
\\ &
\qquad
\left\{
2(2u+s+m+1)s
+x^2+y^2-(m+u+s+1)^2-(u+s)^2
\right\}
\\ &
=\sum_{s=0}^{m}
\dfrac{(-1)^{s}\binom{m}s}{\binom{ m+u+s + \frac12}{m + 1}}
P(x,m+u+s)\dfrac{P(y,u + m)}{P(y,u + s)}
\left\{
x^2+y^2
-(m+u+1)^2
-u^2
\right\}
\\ &
=\dfrac1x\left\{
x^2 + y^2 - (m + u + 1)^2 - u^2\right\}
h_{1,m, u}.
\qedhere
\end{align*}
\end{proof}

\begin{lem}\label{lem1}
The following equation holds.
$$
\sum_{s=0}^m
\dfrac{(-1)^s\binom{m}{s}}
{\binom{m+s+u+\frac12}{m+1}}
\binom{
k+m+u+s+\frac12
}{2m}
=\dfrac{(m+1)(k)_{m}^\downarrow
(u+k+\frac32)_{m}^\uparrow}
{(u+\frac12)_{2m+1}^\uparrow}.
$$
\end{lem}
\begin{proof}
We denote the left hand side of the equation in the assertion as $L$. 
Note that 

\begin{align*}
L&=
\sum_{s=0}^m
\dfrac{(-1)^s\binom{m}{s}}
{\binom{m+s+u+\frac12}{m+1}}
\binom{
k+m+u+s+\frac12
}{2m}
=
\sum_{s=0}^m
\dfrac{(-1)^s(m)_s^\downarrow
(m+1)!
(k+m+u+s+\frac12)_{2m}^\downarrow
}{s!
(m+s+u+\frac12)_{m+1}^\downarrow
(2m)!
}
\\
&
=
\dfrac{(k+m+u+\frac12)_{2m}^\downarrow(m+1)!
}{(u+\frac12)_{m+1}^\uparrow
(2m)!}
\sum_{s=0}^m \dfrac{(-m)_s^\uparrow
(k+m+u+\frac32)_{s}^\uparrow (u+\frac12)_{s}^\uparrow
}{s! (k-m+u+\frac32)_{s}^\uparrow (m+u+\frac32)_{s}^\uparrow}.
\end{align*}
Now, by Pfaff-Saalsch\"utz theorem, we have 
\begin{align*}
&\sum_{s=0}^m \dfrac{(-m)_s^\uparrow
(k+m+u+\frac32)_{s}^\uparrow (u+\frac12)_{s}^\uparrow
}{s! (k-m+u+\frac32)_{s}^\uparrow (m+u+\frac32)_{s}^\uparrow}
=
\left.{}_{3}F_{2}\genfrac(.{0pt}{}%
{\,{-m}\,,\,{u+\frac12}\,,\,{k+m+u+\frac32}\,}
{\,{m+u+\frac32}\,,\,{k-m+u+\frac32}\,};1\right)
\\ &
=\dfrac{
\left(
(m+u+\frac32)-(u+\frac12)
\right)_{m}
\left(
(m+u+\frac32)-(k+m+u+\frac32)
\right)_{m}
}{
(m+u+\frac32)_{m}
\left(
(m+u+\frac32)-(u+\frac12)-(k+m+u+\frac32)
\right)_{m}
}
\\ &
=
\dfrac{
(m+1)_{m}^\uparrow
(-k)_{m}^\uparrow
}{
(m+u+\frac32)_{m}^\uparrow
(-k-u-\frac12)_{m}^\uparrow
}
=
\dfrac{
(m+1)_{m}^\uparrow
(k)_{m}^\downarrow
}{
(m+u+\frac32)_{m}^\uparrow
(k+u+\frac12)_{m}^\downarrow
}.
\end{align*}
Therefore we conclude that 
\begin{align*}
L&
=
\dfrac{(k+m+u+\frac12)_{2m}^\downarrow(m+1)!
\cdot
(m+1)_{m}^\uparrow(k)_{m}^\downarrow
}{
(u+\frac12)_{m+1}^\uparrow
(2m)!
\cdot
(m+u+\frac32)_{m}^\uparrow
(k+u+\frac12)_{m}^\downarrow
}
=
\dfrac{(m+1)
(k)_{m}^\downarrow
(u+k+\frac32)_{m}^\uparrow
}{
(u+\frac12)_{2m+1}^\uparrow
}.
\qedhere
\end{align*}
\end{proof}

\begin{prop}\label{prop2}
There exists a constant $A_{m,u}$ satisfying the following relations.
\begin{align*}
h_{1,m,u}(x;y)
&\equiv
\phantom{-}
A_{m,u}
(x  -m - u)_{3m + 2u + 1}^\uparrow
(x +  1/2)_{m}^\uparrow
\\
& \equiv
-
A_{m,u}
(y  -m-u)_{3m + 2u + 1}^\uparrow
(y  +1/2)_{m}^\uparrow
\mod{(x+y+m)}.
\end{align*}
\end{prop}
\begin{proof}
For $0\leq s\leq m$, we see that 
$P(x,m+u+s)\dfrac{P(y,u+m)}{P(y,u+s)}$ is divisible by 
\begin{align*}
&
P(x,m+u+s)
(y-u-m)_{m-s}^\uparrow
\equiv
P(x,m+u+s)
(-x-u-2m)_{m-s}^\uparrow
\mod{(x+y+m)}
\\
&\quad 
=
(-1)^{m-s}
P(x,m+u+s)
(x+2m+u)_{m-s}^\downarrow
=(-1)^{m-s}(x+2m+u)_{3m+2u+s+1}^\downarrow
\\
&\quad
=(-1)^{m-s}
(x-m-u)_{3m+2u+1}^\uparrow
(x-m-u-1)_{s}^\downarrow.
\end{align*}
It follows that 
$$h_{1,m,u}(x;y)\in((x-m-u)_{3m+2u+1}^\uparrow)+(x+y+m).
\eqno{(5)}$$

By applying Lemma \ref{lem1}, we obtain 
\begin{align*}
\lefteqn{
P\Bigl(-m+k+\frac12,u + m\Bigr)^{-1}
h_{1,m,u}\Bigl(-k-\frac12;-m+k+\frac12\Bigr)}
\\
&=
\sum_{s=0}^m
\dfrac{(-1)^{s}\binom{m}{s}}{\binom{ m+u+s + \frac12}{m + 1}}
\dfrac{P(-k-\frac12,m+u+s)
}{P(-m+k+\frac12,u + s)}
=
\sum_{s=0}^m
\dfrac{(-1)^{s+1}\binom{m}{s}}{\binom{ m+u+s + \frac12}{m + 1}}
\dfrac{P(k+\frac12,m+u+s)
}{P(-m+k+\frac12,u + s)}
\\
&=
\sum_{s=0}^m
\dfrac{(-1)^{s+1}\binom{m}{s}}{\binom{ m+u+s + \frac12}{m + 1}}
\dfrac{(k+m+u+s+\frac12)_{2m+2u+2s+1}^\downarrow
}{(k-m+u+s+\frac12)_{2u + 2s+1}^\downarrow}
=
\sum_{s=0}^m
\dfrac{(-1)^{s+1}\binom{m}{s}
(k+m+u+s+\frac12)_{2m}^\downarrow
}{\binom{ m+u+s + \frac12}{m + 1}}
\\ 
&=-\dfrac{(m+1)(2m)!(k)_{m}^\downarrow
(u+k+\frac32)_{m}^\uparrow}
{(u+\frac12)_{2m+1}^\uparrow}.
\end{align*}
In particular, for integers $0\leq k<m$, we have 
$$h_{1,m,u}\Bigl(-k-\frac12;-m+k+\frac12\Bigr)=0.$$
It implies that 
$$h_{1,m,u}(x;y)\in((x+1/2)_{m}^\uparrow)+(x+y+m).
\eqno{(6)}$$
By (5) and (6), we conclude that 
$$h_{1,m,u}(x;y)
\in((x+1/2)_{m}^\uparrow(x-m-u)_{3m+2u+1}^\uparrow)+(x+y+m).$$
Since 
$\deg h_{1,m,u}(x;y)\leq 4m+2u+1$, there exists a 
constant $A_{m,u}$ such that 
$$h_{1,m,u}(x;y)\equiv
A_{m,u}
(x-m-u)_{3m+2u+1}^\uparrow
(x+1/2)_{m}^\uparrow
\mod{(x+y+m)}.
\eqno{(7)}$$
Recall that we have 
\begin{align*}
&
(x-m-u)_{3m+2u+1}^\uparrow
(x+1/2)_{m}^\uparrow
\equiv
(-y-2m-u)_{3m+2u+1}^\uparrow
(-y-m+1/2)_{m}^\uparrow
\mod{(x+y+m)}
\\
&\quad
=
(-y+m+u)_{3m+2u+1}^\downarrow
(-y-1/2)_{m}^\downarrow
=-
(y-m-u)_{3m+2u+1}^\uparrow
(y+1/2)_{m}^\uparrow.
\end{align*}
Thus (7) implies 
$$h_{1,m,u}(x;y)\equiv
-A_{m,u}
(y-m-u)_{3m+2u+1}^\uparrow
(y+1/2)_{m}^\uparrow
\mod{(x+y+m)}.
\eqno{\qed}
$$
\renewcommand{\qed}{}
\end{proof}
\end{subsection}
\begin{subsection}{The case $\ell>1$.}
Next we examine the properties of 
$h_{\ell,m,u}(x;\bm{y})$ with $\ell>1$. 
For this purpose, we define an auxiliary function 
$\Phi_{m,k,v}$.  For $m,k,v\in\bZ_{\geq0}$, we define 
$$
\Phi_{m,k,v}(y,z)=\Phi_{m,k,v}=
\sum_{t=0}^m
\binom{m}{t}\binom{m}{k-t}
\dfrac{P(y,v+m)P(z,v+m+t)}{P(y,v+t)P(z,v+k)}.
$$
\begin{lem}\label{Phi}

\ \medskip

\item[(1)]
The following recurrence formula for $k$ holds.
$$\Phi_{m, k, v + 1} - \Phi_{m, k, v}
+ (k + 1)(2v + k + m + 2)\Phi_{m, k + 1, v}=0.$$
\item[(2)]
$\Phi_{m,k,v}$ is symmetric in $y$ and $z$, namely, 
$\Phi_{m,k,v}(y,z)=\Phi_{m,k,v}(z,y)$.
\end{lem}
\begin{proof}

\ \medskip

\item{(1)}
We set 
$$\Phi_{m,k,v,r}=
\sum_{t=r}^m
\binom{m}{t}\binom{m}{k-t}
\dfrac{P(y,v+m)P(z,v+m+t)}{P(y,v+t)P(z,v+k)}$$
and show the formula 
\begin{align*}
&\Phi_{m, k, v + 1,r} - \Phi_{m, k, v,r}
+ (k + 1)(2v + k + m + 2)\Phi_{m, k + 1, v,r}
\\
&\qquad
= r(2v + r + m + 1)
\binom{m}{r}\binom{m}{k-r+1 }
\dfrac{
P(y,  v + m)P(z,  v + m + r)
}{
P(y, v  + r)P(z, v + k +1)
}
\tag{8}
\end{align*}
by descending induction on $r$. 
For $r>m$, $(8)$ is clear since both sides are $0$. 

Assume (8) holds for $r+1$.  Then it follows that 
\begin{align*}
&\Phi_{m, k, v + 1,r} - \Phi_{m, k, v,r}
+ (k + 1)(2v + k + m + 2)\Phi_{m, k + 1, v,r}
\\ &
=
(r+1)(2v + r + m + 2)
\binom{m}{r+1}\binom{m}{ k-r}
\dfrac{
P(y,  v + m)P(z,  v + m + r+1)
}{
P(y, v +  r+1)P(z, v + k + 1)
}
\\&
\phantom{=}
+
\binom{m}{r}\binom{m}{k-r}
\dfrac{P(y,v+m+1)P(z,v+m+r+1)}{P(y,v+r+1)P(z,v+k+1)}
- \binom{m}{r}\binom{m}{k-r}
\dfrac{P(y,v+m)P(z,v+m+r)}{P(y,v+r)P(z,v+k)}
\\&
\phantom{=}
+ (k + 1)(2v + k + m + 2)
\binom{m}{r}\binom{m}{k-r+1}
\dfrac{P(y,v+m)P(z,v+m+r)}{P(y,v+r)P(z,v+k+1)}
\\ &
=\binom{m}{r}\binom{m}{ k-r}
\dfrac{P(y,  v + m)P(z,  v + m + r)
}{
P(y, v +  r+1)
P(z,v+k+1)
}
\left\{
(m-r)
(2v + r + m + 2)
(z^2-(v+m+r+1)^2)
\right.
\\&
\phantom{=}
+
(y^2-(v+m+1)^2)
(z^2-(v+m+r+1)^2)
-
(y^2-(v+r+1)^2)
(z^2-(v+k+1)^2)
\\&
\phantom{=}
\left.
+ (k + 1)(2v + k + m + 2)
\dfrac{m-k+r}{k-r+1}
(y^2-(v+r+1)^2)
\right\}
\\ &
=\binom{m}{r}\binom{m}{ k-r}
\dfrac{P(y,  v + m)P(z,  v + m + r)
}{
P(y, v +  r+1)
P(z,v+k+1)
}
\dfrac{r(2v+r+m+1)(m-k+r)}
{k-r+1}
(y^2-(v+r+1)^2)
\\ &
=r(2v + r + m + 1)
\binom{m}{r}\binom{m}{k-r+1 }
\dfrac{
P(y,  v + m)P(z,  v + m + r)
}{
P(y, v  + r)P(z, v + k +1)
},
\end{align*}
which verifies $(8)$ for $r$.
Thus $(8)$ holds for all $r$.

By setting $r=0$ in $(8)$, we have 
the desired formula in the assertion.

\medskip

\item[(2)]
It is clear for $k=0$ since 
$$\Phi_{m, 0 , v} =
\dfrac{P(y,v+m)P(z,v+m)}{P(y,v)P(z,v)}.$$
For $k>0$, the assertion follows immediately by 
induction on $k$ by virtue of (1).
\end{proof}
\begin{prop}\label{symm}
$h_{\ell,m,u}(x;y_1,\dotsc,y_\ell)$ is symmetric in 
$y_1,\dotsc,y_\ell$, and satisfies 
a recurrence formula as below for $1\leq i\leq \ell$.
$$
h_{\ell,m,u}(x;y_1,\dotsc,y_\ell) =
\sum_{k=0}^m
(-1)^k\binom{m}{k}\dfrac{P(y_i,u+m)}{P(y_i,u+k)}
h_{\ell-1,m,u+k}(x;y_1,\dotsc,\hat{y_{i}},\dotsc,y_{\ell}).
$$
\end{prop}
\begin{proof}
By definition of $h_{\ell,m,u}$, we have 
$$
h_{\ell,m,u}(x;y_1,\dotsc,y_\ell) =
\sum_{k=0}^m
(-1)^k\binom{m}{k}\dfrac{P(y_1,u+m)}{P(y_1,u+k)}
h_{\ell-1,m,u+k}(x;y_2,\dotsc,y_{\ell}),
\eqno{(9)}
$$
i.e., the recurrence formula for $i=1$ holds. 
For $\ell\geq2$, by applying $(9)$ twice, we have 
\begin{align*}
&h_{\ell,m,u}
(x;y_1,\dotsc,y_\ell) 
\\
&\quad 
=
\sum_{a=0}^m\sum_{b=0}^{m}
(-1)^{a+b}\binom{m}{a}\binom{m}{b}
\dfrac{P(y_1,u+m)P(y_2,u+m+a)}{P(y_1,u+a)P(y_2,u+a+b)}
h_{\ell-2,m,u+a+b}(x;y_3,\dotsc,y_{\ell})
\\ &\quad  =
\sum_{k=0}^{2m}
\sum_{a=0}^m
(-1)^{k}\binom{m}{a}\binom{m}{k-a}
\dfrac{P(y_1,u+m)P(y_2,u+m+a)}{P(y_1,u+a)P(y_2,u+k)}
h_{\ell-2,m,u+k}(x;y_3,\dotsc,y_{\ell})
\\ &\quad  =
\sum_{k=0}^{2m}
\Phi_{m,k,u}(y_1,y_2)
h_{\ell-2,m,u+k}(x;y_3,\dotsc,y_{\ell}).
\end{align*}
Thus Lemma \ref{Phi} implies that 
$h_{\ell,m,u}$ is symmetric in $y_1$ and $y_2$.
Now we prove 
the symmetricity of $y_1,\dots,y_\ell$ 
in $h_{\ell,m,u}$ by induction on $\ell$.  
For $\ell=1$, we have nothing to prove. 
For $\ell\geq2$, the formula $(9)$ and 
induction hypothesis imply that 
$h_{\ell,m,u}$ is symmetric in $y_2,\dots,y_\ell$.  
Since it is also symmetric in $y_1$ and $y_2$ by the above observation, 
we conclude that 
$h_{\ell,m,u}$ is symmetric in all $y_1,\dots,y_\ell$.  
In particular, the recurrence formula holds for all $i$.
\end{proof}
\end{subsection}
\end{section}
\begin{section}{Proof of the main theorem}
We prove Theorem \ref{Main}.  
The main tool is Saito's criterion.
\begin{thm}[Saito's criterion, \cite{Sa}]
Let $\cA\subset\bC^\ell$ be a central arrangement defined by 
$F:=\prod_{H \in \cA} \alpha_H$, where 
$\alpha_H$ is a defining linear form for a hyperplane $H$.  
The homogeneous derivations $\theta_1,\dotsc,\theta_\ell \in D(\cA)$ 
form a basis for $D(\cA)$ if and only if 
$\det 
\left[
\theta_j(x_i)
\right]
=cF$ for some $c \in \bC^\times$. In particular, 
$\sum_{i=1}^\ell  \deg \theta_i=|\cA|=\deg F$.
\end{thm}

\noindent I)\quad
We show that $\delta_{\ell,m,u}\in D(c\!\,\cat(B_\ell,m))$ 
for $u\geq0$. 
Since 
$c\!\,\cat(B_\ell,m)$ is defined by 
$z\widetilde{\Psi}_{\ell,m}=0$, 
the conditions for 
$\delta_{\ell,m,u}\in D(c\!\,\cat(B_\ell,m))$ 
are 
$$
\delta_{\ell,m,u}(x_a)\in
(\widetilde{P}(x_a,m)),\quad
\delta_{\ell,m,u}(x_b\pm x_c)\in
(\widetilde{P}(x_b\pm x_c,m))
\quad
(1\leq a\leq \ell, 1\leq b<c\leq \ell),
$$
%where $\widetilde{P}(\ast)$ denotes the homogenization of ${P}(\ast)$.  
Equivalently, we should prove 
$$
\begin{array}{lll}
(1) & F_{\ell,m,u,a}\in({P}(x_a,m)) & (1\leq a\leq \ell),
 \\
(2) & F_{\ell,m,u,b}\pm F_{\ell,m,u,c}\in({P}(x_b\pm x_c,m)) & (1\leq b<c\leq \ell).
\end{array}
$$
The proof of this part is separated into three steps.

\medskip

\noindent{\it Step 1}.\quad
The condition (1) is clear from the definition of $F_{\ell,m,u,a}$ since 
$$
F_{\ell,m,u,a}=
h_{\ell-1,m,u}(y_a;y_1,\dotsc,\hat{y_{a}},\dotsc,y_{\ell})
\in\sum_{\ell}(P(x_a,m+u+S_{\ell-1}))
\subset(P(x_a,m)).
$$

\noindent{\it Step 2}.\quad
We investigate the condition (2) for 
$\ell=2$.  We observe by Proposition \ref{prop2} that 
\begin{align*}
\lefteqn{F_{2,m,u,1}+ F_{2,m,u,2}
=h_{1,m,u}(x_1;x_2)+h_{1,m,u}(x_2;x_1)
}
\\
&\equiv
\left\{
A_{m,u}
(x_1  -m - u)_{3m + 2u + 1}^\uparrow
(x_1 +  1/2)_{m}^\uparrow
\right\}
+
\left\{
-
A_{m,u}
(x_1  -m-u)_{3m + 2u + 1}^\uparrow
(x_1  +1/2)_{m}^\uparrow
\right\}
\\
&=0
\mod{(x_1+x_2+m)}.
\end{align*}
%Namely, we have $$F_{2,m,u,1}+ F_{2,m,u,2}\in((x_1+x_2+m)).$$
We also observe by Proposition \ref{prop1} that 
\begin{align*}
&\dfrac{2u - 1}{2m + 4}
(F_{2,m+1,u-1,1}+ F_{2,m+1,u-1,2})
\\
&\quad=\left\{
x^2 + y^2 - (m + u + 1)^2 - u^2\right\}
(F_{2,m,u,1}+ F_{2,m,u,2})
- 2
(F_{2,m,u+1,1}+ F_{2,m,u+1,2})
.
\end{align*}
Combining these observations, it follows by induction on $m$ that 
$$
F_{2,m,u,1}+ F_{2,m,u,2}\in((x_1+x_2)_{m+1}^\uparrow).
$$
Remember that $h_{1,m,u}(x;y)$ is odd function in $x$ and even 
function in $y$.
As $h_{1,m,u}(-x;-y)=-h_{1,m,u}(x;y)$, above formula implies 
$$
-(F_{2,m,u,1}+ F_{2,m,u,2})
\in((x_1+x_2)_{m+1}^\downarrow), 
$$
and these two formulae verify 
$$
F_{2,m,u,1}+ F_{2,m,u,2}\in(P(x_1+x_2,m)).
$$
We also observe that 
\begin{align*}
F_{2,m,u,1}- F_{2,m,u,2}
&=h_{1,m,u}(x_1;x_2)-h_{1,m,u}(x_2;x_1)
\\
&=h_{1,m,u}(x_1;-x_2)+h_{1,m,u}(-x_2;x_1)
\in(P(x_1-x_2,m)).
\end{align*}
Therefore the condition (2) for $\ell=2$ is proven. 

\medskip

\noindent{\it Step 3}.\quad
We prove the condition (2) 
%for $\ell\geq2$ 
by induction on $\ell$.  
Take $d\in \{1,\dotsc,\ell\}\setminus\{b,c\}$ and 
set $\bm{x}'=(x_i\mid i\neq b,c,d)$.  
Then, by Proposition \ref{symm} and 
induction hypothesis, we see that 
\begin{align*}
\lefteqn{F_{\ell,m,u,b}\pm F_{\ell,m,u,c}
=h_{\ell-1,m,u}(x_b;x_d,x_c,\bm{x}')
\pm h_{\ell-1,m,u}(x_c;x_d,x_c,\bm{x}')
}\\
&
=
\sum_{k=0}^m
(-1)^k\binom{m}{k}\dfrac{P(x_d,u+m)}{P(x_d,u+k)}
\Bigl\{
h_{\ell-2,m,u+k}(x_b;x_c,\bm{x}')
\pm h_{\ell-2,m,u+k}(x_c;x_b,\bm{x}')
\Bigr\}
\\ &
=
\sum_{k=0}^m
(-1)^k\binom{m}{k}\dfrac{P(x_d,u+m)}{P(x_d,u+k)}
\Bigl\{
(F_{\ell-1,m,u+k,b}\pm F_{\ell-1,m,u+k,c})
\Bigr|_{x_d=x_\ell}\Bigr\}
%Here the variables for $F_{\ell-1,m,u+k,c})$ are 
%$\{x_i\mid 1\leq i\leq \ell,\ i\neq d\}$.
\in(P(x_b\pm x_c,m))
\end{align*}
Therefore, (2) holds for all $\ell$.
It completes the proof for 
$\delta_{\ell,m,u}\in D(c\!\,\cat(B_\ell,m))$ for $u\geq0$. 

\bigskip

\noindent II)\quad
We apply Saito's criterion to the derivations 
$\left\{\delta_E\right\}\cup
\left\{
\delta_{\ell,m,u}\mid 0\leq u<\ell
\right\}$.
We set 
$$
Q_{m,\ell}=\det \left[
\delta_E,
\delta_{\ell,m,0},
\delta_{\ell,m,1},
\dotsc,
\delta_{\ell,m,\ell-1}
\right].
$$
Since all entries above belong to 
$D(c\!\,\cat(B_\ell,m))$, we have $Q_{m,\ell}
\in(z\widetilde{\Psi}_{m,\ell})$.
Note that 
%$$
%\deg\delta_{\ell,m,u}=\deg h_{\ell-1,m,u}
%=2m+2u+S_{\ell-1}+\sum_{i=1}^{\ell-1}2(m-s_i)
%=2m+2u+2S_{\ell-1}+1+\sum_{i=1}^{\ell-1}2(m-s_i)
%=2m\ell+2u+1,$$
%and hence 
$$
\deg\delta_E+\sum_{u=0}^{\ell-1}\deg\delta_{\ell,m,u}
\leq
1+\sum_{u=0}^{\ell-1}(2m\ell+2u+1)
=1+(2m+1)\ell^2.
$$
On the other hand, we have 
$$
\deg z\widetilde{\Psi}_{l,m}
=1+(2m+1)\ell+2\cdot(2m+1)\cdot\binom{\ell}2
=1+(2m+1)\ell^2.
$$
It follows that there exists a constant $C\in\bC$ such that 
$Q_{m,\ell}=Cz\widetilde{\Psi}_{m,\ell}$.  
%Observe that the highest exponent of monomials appearing in $Q_{m,\ell}$ 
%in lexicographic order is 
%$$M_0=z\prod_{i=1}^\ell x_i^{(2\ell-2i+1)(2m+1)}.$$
Comparing the coefficients of 
$z\prod_{i=1}^\ell x_i^{(2\ell-2i+1)(2m+1)}$ 
in both sides, we see that 
$$
\prod_{i=0}^{\ell-1}
\dfrac{(-1)^{im}}{\binom{ (i+1)(m+1)- \frac12}{m + 1}}
=C\cdot1.
$$
It follows that $C\neq0$.
Therefore 
$\left\{\delta_E\right\}\cup
\left\{
\delta_{\ell,m,u}\mid 0\leq u<\ell
\right\}$
gives a basis of $D(c\!\,\cat(B_\ell,m))$.
\end{section}

\bigskip

\bigskip

\end{document}